\documentclass[a4paper,12pt]{article}
\usepackage{amssymb,amsthm,amsmath,latexsym}

\newtheorem{thm}{Theorem}[section]
\newtheorem{pro}[thm]{Proposition}
\newtheorem{lem}[thm]{Lemma}

\newtheorem*{thmA}{Theorem A}
\newtheorem*{thmB}{Theorem B}
\newtheorem*{thmC}{Theorem C}

\def\l{\langle}
\def\r{\rangle}

\date{}

\title{\normalsize{\bf ON VARIETIES OF GROUPS SATISFYING\\ AN ENGEL TYPE IDENTITY}}

\author{\small{\textsc{Pavel Shumyatsky\begin{footnote}{This work was carried out during the first author's visit to the University of Salerno. He would like to thank the Department of Mathematics for hospitality and GNSAGA (INdAM) for financial support.}\end{footnote}}}\\
\small{Department of Mathematics, University of Brasilia}\\
\small{Brasilia-DF, 70910-900 Brazil}\\
\small{E-mail: pavel@unb.br}\\
[10pt]
\small{\textsc{Antonio Tortora} and \textsc{Maria Tota}}\\
\small{Dipartimento di Matematica, Universit\`a di Salerno}\\
\small{Via Giovanni Paolo II, 132 - 84084 - Fisciano (SA), Italy}\\
\small{E-mail: antortora@unisa.it, mtota@unisa.it}}

\begin{document}
\maketitle

\begin{abstract} Let $m,n$ be positive integers, $v$ a multilinear commutator word and $w=v^m$. Denote by $v(G)$ and $w(G)$ the verbal subgroups of a group $G$ corresponding to $v$ and $w$, respectively. We prove that the class of all groups $G$ in which the $w$-values are $n$-Engel and $w(G)$ is locally nilpotent is a variety (Theorem A). Further, we show that in the case where $m$ is a prime-power the class of all groups $G$ in which the $w$-values are $n$-Engel and $v(G)$ has an ascending normal series whose quotients are either locally soluble or locally finite is a variety (Theorem B). We examine the question whether the latter result remains valid with $m$ allowed to be an arbitrary positive integer. In this direction, we show that if $m,n$ are positive integers, $u$ a multilinear commutator word and $v$ the product of 896 $u$-words, then the class of all groups $G$ in which the $v^m$-values are $n$-Engel and the verbal subgroup $u(G)$ has an ascending normal series whose quotients are either locally soluble or locally finite is a variety (Theorem C).
\\

\end{abstract}

\section{Introduction}

A variety is a class of groups defined by equations. More precisely, if $W$ is a set of words, the class of all groups satisfying the laws $W\equiv 1$ is called the variety determined by $W$. By a well-known theorem of Birkhoff \cite[2.3.5]{Rob}, varieties are precisely classes of groups closed with respect to taking subgroups, quotients and Cartesian products of their members. Some interesting varieties of groups have been discovered in the context of the Restricted Burnside Problem solved in the affirmative by Zelmanov \cite{ze1,ze2}.

It is well-known that the solution of the Restricted Burnside Problem is equivalent to each of the following statements.

\begin{itemize}
\item[$(i)$] The class of locally finite groups of exponent $n$ is a variety.
\item[$(ii)$] The class of locally nilpotent groups of exponent $n$ is a variety.
\end{itemize}

The equivalence of the two above results follows from the famous Hall-Higman reduction theorem \cite{hahi}. Recall that a group is said to locally have some property if all its finitely generated subgroups have that property. A number of varieties of (locally nilpotent)-by-soluble groups were presented in \cite{shu1, shu2}.

The solution of the Restricted Burnside Problem strongly impacted our understanding of Engel groups. An element $x$ of a group $G$ is called a {\em (left) Engel element} if for any $g\in G$ there exists $n=n(x,g)\geq 1$ such that $[g,_n x]=1$, where the commutator $[g,_n x]$ is defined recursively by the rule
$$[g,_n x]=[[g,_{n-1}x],x]$$
starting with $[g,_0x]=g$ and $[g,_1x]=[g,x]=g^{-1}x^{-1}gx$. If $n$ can be chosen independently of $g$, then $x$ is a (left) $n$-Engel element. A group $G$ is called $n$-Engel if all elements of $G$ are $n$-Engel. In \cite{ze0} Zelmanov remarked that the eventual solution of the Restricted Burnside Problem would imply that the class of locally nilpotent $n$-Engel groups is a variety (see also Wilson \cite{w}). Recently groups with $n$-Engel word-values were considered \cite{BSTT,bs,STT2,STT3}.

If $w$ is a word in variables $x_1,\dots,x_k$ we think of it primarily as a function of $k$ variables defined on any given group $G$. We denote by $w(G)$ the verbal subgroup of $G$ generated by the values of $w$. The word $w$ is called {\em multilinear commutator} (of weight $s$) if it has a form of a multilinear Lie monomial (in precisely $s$ independent variables). Particular examples of multilinear commutators are the lower central words $\gamma_k$,
defined by
\[
\gamma_1=x_1,
\qquad
\gamma_k=[\gamma_{k-1},x_k]=[x_1,\ldots,x_k],
\quad
\text{for $k\ge 2$;}
\]
and the derived words $\delta_k$, on $2^k$ variables, which are defined by
\[
\delta_0=x_1,
\quad
\delta_k=[\delta_{k-1}(x_1,\ldots,x_{2^{k-1}}),\delta_{k-1}(x_{2^{k-1}+1},\ldots,x_{2^k})],
\quad
\text{for $k\ge 1$.}
\]
The verbal subgroups $\gamma_k(G)$, corresponding to $\gamma_k$, and $G^{(k)}$, corresponding to $\delta_k$, are the $k$-th term of the lower central series of $G$ and the $k$-th derived subgroup of $G$, respectively.

The goal of the present article is to prove the following theorems.

\begin{thmA}
Let $m,n$ be positive integers, $v$ a multilinear commutator word and $w=v^m$. Then the class of all groups $G$ in which the $w$-values are $n$-Engel and the verbal subgroup $w(G)$ is locally nilpotent is a variety.
\end{thmA}

\begin{thmB}
Let $m$ be a prime-power, $n$ a positive integer, $v$ a multilinear commutator word and $w=v^m$. Then the class of all groups $G$ in which the $w$-values are $n$-Engel and the verbal subgroup $v(G)$ has an ascending normal series whose quotients are either locally soluble or locally finite is a variety.
\end{thmB}

We do not know whether the hypothesis that $m$ is a prime-power is essential in Theorem B. Our next result is an evidence supporting the conjecture that the theorem remains correct with $m$ allowed to be any positive integer. Given a positive integer $s$ and a word $u=u(x_1,\dots,x_k)$, we call the word $$u(x_{11},\dots,x_{1k})u(x_{21},\dots,x_{2k})\cdots u(x_{s1},\dots,x_{sk})$$ of $ks$ variables the product of $s$ $u$-words. It is clear that for any word $u$, positive integer $s$ and a group $G$, the verbal subgroup $u(G)$ coincides with the verbal subgroup $v(G)$, where $v$ is the product of $s$ $u$-words.

\begin{thmC}
Let $m,n$ be positive integers, $u$ a multilinear commutator word and $v$ the product of $896$ $u$-words. Set $w=v^m$. Then the class of all groups $G$ in which the $w$-values are $n$-Engel and the verbal subgroup $u(G)$ has an ascending normal series whose quotients are either locally soluble or locally finite is a variety.
\end{thmC}

The constant 896 in the statement comes from the results of Nikolov and Segal on commutator width of finite groups \cite{nise}. The proof of Theorem C uses special properties of finite groups in which every product of 896 $u$-values has order dividing $m$ (see Lemma \ref{896} in Section 4 for details).

An inspection of the proofs of Theorems B and C shows that whenever a group $G$ belongs to one of the varieties considered in those theorems, the verbal subgroup $v(G)$ is actually (locally nilpotent)-by-(locally finite) and hence locally (nilpotent-by-finite). This observation leads to some alternative characterizations of the varieties. This is discussed in Section 5.

Most important ingredients of the proofs presented in the article originate from the techniques created by Zelmanov in his solution of the Restricted Burnside Problem. The classification of finite simple groups is used as well.

Throughout the paper we write, ``$\{a,b,\dots\}$-bounded'' to abbreviate ``bo\-unded from above in terms of  $a,b,\dots$ only''.

\section{Proof of Theorem A}

The proof of Theorem A will be quite short.

Let $m,n$ be positive integers, $v$ a multilinear commutator word and $w=v^m$. We denote by $\mathcal W=\mathcal W(m,n,v)$ the class of all groups $G$ in which the $w$-values are $n$-Engel and the verbal subgroup $w(G)$ is locally nilpotent.

An important role in the proof will be played by the following proposition taken from \cite{BSTT}.

\begin{pro}\label{14}
Let $G$ be a residually finite group in which all $w$-values are $n$-Engel. If $G$ is generated by finitely many Engel elements, then $G$ is nilpotent.
\end{pro}

\begin{lem}\label{bou} There exists a number $\epsilon=\epsilon(d,m,n,v)$ depending only on $d,m,n,v$ such that if $G\in\mathcal W$ is a nilpotent group generated by $d$ elements which are $n$-Engel, then the nilpotency class of $G$ is at most $\epsilon$.
\end{lem}

\begin{proof}
Suppose that this is false. Then there exists an infinite sequence $(G_i)_{i\geq 1}$ of nilpotent groups satisfying the hypotheses of the lemma such that the nilpotency class of $G_i$ tends to infinity as $i$ tends to infinity. In  each group $G_i$ we choose $d$ generators $x_{i1},\dots,x_{id}$ which are $n$-Engel. Here the elements $x_{i1},\dots,x_{id}$ are not necessarily pairwise distinct. Let $C$ be the Cartesian product of the groups $G_i$ and let $y_1,\dots,y_d$ be the elements of $C$ such that the $i$-th component of $y_j$ is equal to $x_{ij}$ for every $i\geq 1$ and $1\leq j\leq d$. Obviously, each of the elements $y_i$ is $n$-Engel in $C$. Let $H$ be the subgroup of $C$ generated by $y_1,\dots,y_d$. Since every group $G_i$ is a homomorphic image of $H$, the subgroup $H$ is not nilpotent. On the other hand, being a finitely generated residually nilpotent group, $H$ is residually finite. Moreover all $w$-values in $H$ are $n$-Engel. Thus, by Proposition \ref{14}, $H$ is nilpotent. This is a contradiction.
\end{proof}

The proof of Theorem A is now an easy consequence of the previous lemma.

\begin{proof}[\bf Proof of Theorem A]
The class $\mathcal W$ is obviously closed with respect to taking subgroups and quotients of its members. Hence, it is sufficient to show that if $G$ is a Cartesian product of groups $G_i$ from $\mathcal W$, then $G\in\mathcal W$. Of course, the $w$-values are $n$-Engel in $G$ so it remains only to prove that the verbal subgroup $w(G)$ is locally nilpotent. Take a finitely generated subgroup $H$ of $w(G)$. There exist finitely many $w$-values $w_1,\dots,w_d$ such that $H\leq\langle w_1,\dots,w_d\rangle$. Put $K=\langle w_1,\dots,w_d \rangle$ and let $\pi_i$ be the projection from $K$ to $w(G_i)$. By Lemma \ref{bou} there exists a constant $\epsilon$ such that the nilpotency class of $\pi_i(K)$ is at most $\epsilon$ for every $i$. Since the intersection of the kernels of $\pi_i$ is trivial, $K$ is nilpotent of class at most $\epsilon$. In particular, $H$ is nilpotent and $w(G)$ is locally nilpotent.
\end{proof}

\section{Proof of Theorem B}

It will be convenient first to prove Theorem B in the particular case where $v=\delta_k$ is a derived word. A subset $X$ of a group $G$ is called commutator-closed if $[x,y]\in X$ whenever $x,y\in X$. The fact that in any group the set of all $\delta_k$-values is commutator-closed will be used without explicit references. Following Zelmanov's solution of the Restricted Burnside Problem a number of results bounding the order of a finite group have been found. One example of such a result is provided by the next lemma.

\begin{lem}\label{fp}
Let $p$ be a prime and $G$ a finite $p$-group generated by a normal commutator-closed set of elements of order dividing $m$. Suppose additionally that $G$ satisfies an identity $f\equiv 1$ and is $d$-generated for some $d\geq 1$. Then the order of $G$ is $\{d,m,f\}$-bounded.
\end{lem}

\begin{proof} The proof can be obtained just reproducing the argument of the final paragraph of the proof of Theorem 1.1 in \cite{shu3}. Arguments of this kind are now considered fairly standard. We therefore omit the details.
\end{proof}

The next lemma is a quantitative version of Lemma 3.4 in \cite{shu3}.

\begin{lem}\label{fn}
Let $G$ be a group generated by a normal commutator-closed set $X$ of elements of finite order dividing $m$. Suppose that $G$ is generated by finitely many elements $x_1,\ldots,x_d\in X$ and let $N$ be a subgroup of $G$ of finite index $r$ such that $X\cap N=1$. Then $G$ is finite and its order is $\{d,m,r\}$-bounded.
\end{lem}

\begin{proof}
Without loss of generality, we assume that $N$ is normal in $G$. For any $x\in X$ and $y\in N$ we remark that $[x^y,x]=1$. This follows from the fact that $[x,y,x]=[x^y,x]$ and so $[x^y,x]\in X\cap N=1$. Therefore the subgroup $\langle x^N\rangle$ is abelian of exponent dividing $m$. In particular, the subgroup $[N,x]$ is abelian of exponent dividing $m$. Since
$[N,x]$ is normal in $N$, it follows that the subgroups of the form $[N,x]$  normalize each other. Let $L$ be the product of all subgroups $[N,x]$, where $x$ ranges through the set of generators $x_1,\ldots,x_d\in X$. Of course, $L$ is nilpotent of class at most $d$ and has finite exponent dividing $m^d$. Moreover, $L=[N,G]$ and in particular $L$ is normal in $G$.

The subgroup $N/L$ is contained in the centre of $G/L$ and so the index of $Z(G/L)$ is at most $r$. By Schur's Theorem \cite[Theorem 10.1.4]{Rob}, the commutator subgroup of $G/L$ is finite of $r$-bounded order and we conclude that $G/L$ is finite of $\{d,m,r\}$-bounded order. It follows that $L$ can be generated by $\{d,m,r\}$-boundedly many elements. We know that $L$ has finite exponent dividing $m^d$. Therefore $L$ is finite of $\{d,m,r\}$-bounded order. The lemma follows.
\end{proof}

Given a prime $p$ and a finite group $G$, we denote by $O_p(G)$ the unique maximal normal $p$-subgroup of $G$. In the proposition that follows we will require a recent result of Guralnick and Malle \cite{GM}: {\em if $G$ is a finite group generated by a normal commutator-closed set of $p$-elements, then either $G$ is a $p$-group or $p=5$ and $G/O_5(G)$ is a direct product of copies of $A_5$, the alternating group of degree $5$}.

\begin{pro}\label{radical}
Let $G$ be a group generated by a normal commutator-closed set $X$ of $p$-elements of order dividing $m$. Suppose that $G$ is generated by $d$ elements from $X$ and satisfies an identity $f\equiv1$. If $G$ has an ascending normal series whose quotients are either locally soluble or locally finite, then $G$ is finite of $\{d,m,f\}$-bounded order.
\end{pro}

\begin{proof}
Assume first that $G$ is finite. Then, by Lemma \ref{fp}, we may also assume that $G$ is not a $p$-group. It follows that $p=5$ and $G/O_5(G)$ is a direct product of copies of $A_5$. Since $G/O_5(G)$ is $d$-generated, the number of subgroups of any given index depends only on $d$ \cite[Theorem 7.2.9]{mhall}. Hence, $G/O_5(G)$ is a product of boundedly many copies of the group $A_5$ and in particular the order of $G/O_5(G)$ is $d$-bounded. Set $Y=X\cap O_5(G)$. Applying Lemma \ref{fn} to $G/\l Y\r$, we obtain that the order of $G/\l Y\r$ is $\{d,m\}$-bounded. Therefore $\l Y\r$ has a $\{d,m\}$-bounded number of generators and so, by Lemma \ref{fp}, its order is $\{d,m,f\}$-bounded. This proves the result in the case where $G$ is finite.

Consider now the general case. Let $R$ be the finite residual of $G$, i.e. the intersection of all subgroups of $G$ of finite index. Since any finite quotient of $G$ has bounded order, the index of $R$ in $G$ is finite. If $R=1$, we are done. Assume $R\neq 1$. Then $R$ is finitely generated and therefore $R$ cannot be a union of an increasing chain of proper subgroups. Let $R=\bigcup_{\alpha<\beta} R_\alpha$ where $1\leq R_1\leq R_2\leq\ldots\leq R_\beta=R$ and $R_{\alpha+1}/R_\alpha$ is either locally soluble or locally finite. Suppose that $\beta$ is a minimal ordinal such that $R_{\beta}=R$. It follows that $\beta$ is not a limit ordinal and hence $R/R_{\beta-1}$ is either finite or soluble. If $R/R_{\beta-1}$ is finite, then so is $G/R_{\beta-1}$ and consequently $R=R_{\beta-1}$, a contradiction. Thus $R/R_{\beta-1}$ is soluble and so $R'R_{\beta-1}\neq R$. Here $R'$ is the commutator subgroup of $R$. In particular $R'<R$ and $G/R'$ is residually finite. We conclude that $G/R'$ is finite, which yields the final contradiction.
\end{proof}

In any group $G$ there exists a unique maximal normal locally nilpotent subgroup $F(G)$ (called the Hirsch-Plotkin radical) containing all normal locally nilpotent subgroups of $G$ \cite[12.1.3]{Rob}. In general, $F(G)$ is a subset of the set $L(G)$ of all (left) Engel elements \cite[12.3.2]{Rob}. However, in \cite{Pl}, Plotkin proved that $F(G)$ coincides with $L(G)$ whenever $G$ has an ascending series whose quotients locally satisfy the maximal condition. We will now prove a related result for the class of groups having an ascending normal series with locally soluble or locally finite quotients.

We will require the following lemma taken from \cite{BSTT}.

\begin{lem}\label{Engel}
Let $G$ be a group generated by a set of Engel elements and $N$ a locally soluble normal subgroup of $G$. Then $[N,G]$ is locally nilpotent.
\end{lem}

\begin{pro}\label{asc}
Let $G$ be a group with an ascending normal series whose quotients are either locally soluble or locally finite. Then the set of all Engel elements of $G$ form a locally nilpotent subgroup.
\end{pro}

\begin{proof}
It is enough to prove that the subgroup $H$ generated by all Engel elements of $G$ is locally nilpotent. Clearly $H$ has an ascending normal series whose quotients are either locally soluble or locally finite. Then Lemma \ref{Engel} can be applied to each locally soluble quotient of the series producing a refined ascending normal series with locally finite or locally nilpotent quotients. Thus $H$ is locally nilpotent, by Plotkin's theorem.
\end{proof}

Let $m$ be a prime-power, $n$ a positive integer, $v$ a multilinear commutator word and $w=v^m$. We denote by $\mathcal V=\mathcal V(m,n,v)$ the class of all groups $G$ in which the $w$-values are $n$-Engel and the verbal subgroup $v(G)$ has an ascending normal series whose quotients are either locally soluble or locally finite.

\begin{lem}\label{bb} There exist numbers $c=c(d,k,m,n)$ and $r=r(d,k,m,n)$ depending only on $d,k,m,n$ such that any subgroup generated by $d$ $\delta_k$-values in a group $G\in\mathcal V(m,n,\delta_k)$ has a normal nilpotent subgroup of class at most $c$ and index at most $r$.
\end{lem}

\begin{proof} Choose a subgroup $H$ generated by $d$ $\delta_k$-values in $G$. Let $N$ be the subgroup generated by all $w$-values of $G$ contained in $H$. Then, by Proposition \ref{radical}, $H/N$ is finite of $\{d,k,m,n\}$-bounded order. It follows that $N$ can be generated by boundedly many elements, say $d'$. We see that $N$ is contained in a subgroup generated by finitely many Engel elements and so, by Proposition \ref{asc}, $N$ is nilpotent. In particular $N$ is residually finite. Let $Q$ be a quotient of $N$ having finite prime-power order. We know that, on the one hand $Q$ is generated by at most $d'$ elements and, on the other hand $Q$ is generated by $n$-Engel elements. Since $Q$ has prime-power order, we are in a position to use the Burnside Basis Theorem \cite[Theorem 5.3.2]{Rob} and conclude that $Q$ is generated by $d'$ elements which are $n$-Engel. Now Lemma \ref{bou} implies that $Q$ is nilpotent of $\{d,k,m,n\}$-bounded class and, consequently, so is $N$.
\end{proof}

\begin{pro}\label{deltak}
The class $\mathcal V=\mathcal V(m,n,\delta_k)$ is a variety.
\end{pro}

\begin{proof}
Clearly the class $\mathcal V$ is closed with respect to taking subgroups and quotients of its members. Let $G$ be a Cartesian product of groups $G_i\in \mathcal V$. All $w$-values are $n$-Engel in $G$ and, by Proposition \ref{asc}, $w(G_i)$ is locally nilpotent for every $i$. This implies that $G$ belongs to the variety addressed in Theorem A and therefore $w(G)$ is locally nilpotent. We will prove that $v(G)/w(G)$ is locally finite, which implies that $G\in\mathcal V$ and $\mathcal V$ is a variety. Any finitely generated subgroup of $v(G)$ is contained in a subgroup generated by finitely many $\delta_k$-values. Therefore it is sufficient to show that if a subgroup $H$ is generated by finitely many (say $d$) $\delta_k$-values, then the image of $H$ in $v(G)/w(G)$ is finite.

Let $H_i$ denote the image of the projection of $H$ on $G_i$. Since $G_i\in \mathcal{V}$, Lemma \ref{bb} shows that each $H_i$ is an extension of a nilpotent group of class $c$ by a finite group of order $r$. Let $K$ be the intersection of all subgroups of $H$ of index $\leq r$. Thus $K$ is nilpotent of class at most $c$, and $H/K$ is finite \cite[Theorem 7.2.9]{mhall}. Now denote by $N$ the subgroup generated by all $w$-values of $G$ contained in $H$. Since $H$ is nilpotent-by-finite, applying Proposition \ref{radical}, we deduce that $H/N$ is finite. It follows that so is $H/H\cap w(G)$, as required.
\end{proof}

The following lemma collects two results that can be found in \cite{shu2}.

\begin{lem}\label{4.1}
Let $G$ be a group and $v$ a multilinear commutator word.
\begin{itemize}
\item[$(i)$] If $v$ is of weight $k$, then every $\delta_k$-value in $G$ is a $v$-value.
\item[$(ii)$] If $G$ is soluble and all $v$-values in $G$ have finite order, then the verbal subgroup $v(G)$ is locally finite.
\end{itemize}
\end{lem}

We are now ready to prove Theorem B.

\begin{proof}[\bf Proof of theorem B]
Let $v$ be a multilinear commutator word and $w=v^m$. Recall that $\mathcal V=\mathcal V(m,n,v)$ is the class of all groups $G$ in which the $w$-values are $n$-Engel and the verbal subgroup $v(G)$ has an ascending normal series whose quotients are either locally soluble or locally finite. We will prove that $\mathcal{V}$ is a variety. Let $G$ be a Cartesian product of groups $G_i\in \mathcal V$. Since $\mathcal V$ is closed with respect to taking subgroups and quotients of its members, the proof will be complete once it is shown that $G\in\mathcal V$. In particular it is sufficient to prove that $v(G)$ has an ascending normal series whose quotients are either locally soluble or locally finite. Of course, all $w$-values are $n$-Engel in $G$. Let $k$ be the weight of $v$. Then, by Lemma \ref{4.1} $(i)$, every $\delta_k$-value in each $G_i$ is a $v$-value. Hence $G_i^{(k)}\leq v(G_i)$ and therefore $G_i^{(k)}$ has an ascending normal series whose quotients are either locally soluble or locally finite. It follows that $G_i\in\mathcal V(m,n,\delta_k)$ for every $i$. The class $\mathcal V(m,n,\delta_k)$ is a variety by Proposition \ref{deltak} and so $G\in\mathcal V(m,n,\delta_k)$. Thus $G^{(k)}$ has an ascending normal series whose quotients are either locally soluble or locally finite and therefore the whole group $G$ has such a series.
\end{proof}

\section{Proof of Theorem C}

Theorem C can be obtained following the same general lines as that of Theorem B. Therefore we will just sketch out some steps in the proof and omit details.

Just as Theorem B, the result easily follows from the particular case where $u=\delta_k$. So without loss of generality we assume that $u=\delta_k$.
The next lemma was established in \cite{shu4}. An error contained in the proof in \cite{shu4} was subsequently corrected in \cite{KS}.

\begin{lem}\label{896} Let $d$, $k$, $m$ be positive integers and $G$ a finite group in which every product of $896$ ${\delta_k}$-values is of order dividing $m$. Assume that $G$ can be generated by $d$ elements $g_1,g_2,\ldots,g_d$ such that each $g_i$ and all commutators of the forms $[g,x]$ and $[g,x,y]$, where $g\in\{g_1,g_2,\ldots,g_d\}$, $x,y\in G$, have orders dividing $m$.
Then the order of $G$ is $\{d,k,m\}$-bounded.
\end{lem}

The following proposition can be deduced from Lemma \ref{896} using a familiar argument.

\begin{pro}\label{radical2} Let $k,m$ be positive integers and $G$ a group in which every product of $896$ $\delta_k$-values has order diving $m$. Assume additionally that $G$ has an ascending normal series whose quotients are either locally soluble or locally finite. Then any subgroup of $G$ that can be generated by $d$ $\delta_k$-values is finite of $\{d,k,m\}$-bounded order.
\end{pro}

\begin{proof} Choose $\delta_k$-values $g_1,\dots,g_d\in G$ and set $H=\langle g_1,\dots,g_d\rangle$. We have $[g,x]=g^{-1}g^x$ and $[g,x,y]=g^{-x}gg^{-y}g^{xy}$. Since $g,g^{-1}$ and all their conjugates are $\delta_k$-values for every $g\in\{g_1,g_2,\ldots,g_d\}$, we observe that the commutators of the forms $[g,x]$ and $[g,x,y]$, where $g\in\{g_1,g_2,\ldots,g_d\}$ and $x,y\in G$, are products of at most 4 $\delta_k$-values. Hence, by the hypothesis, all commutators of the forms $[g,x]$ and $[g,x,y]$ have orders dividing $m$. Thus, if $H$ is finite, by Lemma \ref{896}, the order of $H$ is $\{d,k,m\}$-bounded. On the other hand, if $H$ is infinite, we get a contradiction as in the proof of Proposition \ref{radical}.
\end{proof}

\begin{proof}[\bf Proof of Theorem C]
This follows applying Proposition \ref{radical2} and arguing as in Lemma \ref{bb} and Proposition \ref{deltak}.
\end{proof}

\section{Concluding remarks}

Recall some of the notation used in Section 3. Let $m$ be a prime-power, $n$ a positive integer, $v$ a multilinear commutator word and $w=v^m$. Then $\mathcal V(m,n,v)$ stands for the class of all groups $G$ in which the $w$-values are $n$-Engel and the verbal subgroup $v(G)$ has an ascending normal series whose quotients are either locally soluble or locally finite. Theorem B states that $\mathcal V(m,n,v)$ is a variety. Assume that the multilinear commutator word $v$ is of weight $k$. By Lemma \ref{4.1} $(i)$ in any group the set of $\delta_k$-values is contained in the set of $v$-values.

We will now have a closer look at the proof of Theorem B. Let $G\in\mathcal V(m,n,v)$. Proposition \ref{asc} shows that $w(G)$ is locally nilpotent. From Proposition \ref{radical} we easily deduce that the image of $G^{(k)}$ in $G/w(G)$ is locally finite. Now $(ii)$ of Lemma \ref{4.1} implies that the image of $v(G)$ in $G/G^{(k)}w(G)$ is locally finite. Thus, $v(G)$ is (locally nilpotent)-by-(locally finite) for every group $G\in\mathcal V(m,n,v)$. On the other hand, suppose that in a group $J$ all $w$-values are $n$-Engel and the verbal subgroup $v(J)$ is (locally nilpotent)-by-(locally finite). Of course, $J\in\mathcal V(m,n,v)$.

Therefore we have the following characterization of the variety $\mathcal V(m,n,v)$.
\medskip

\noindent {\it The variety $\mathcal V(m,n,v)$ is precisely the class of all groups $G$ in which the $w$-values are $n$-Engel and the verbal subgroup $v(G)$ is (locally nilpotent)-by-(locally finite).}
\medskip

Let us continue the inspection of the proof of Theorem B. It is easy to slightly generalize Proposition \ref{asc} and show that if $G$ is a group in which every finitely generated subgroup has an ascending normal series whose quotients are either locally soluble or locally finite, then the set of all Engel elements of $G$ form a locally nilpotent subgroup. It is even more obvious that Proposition \ref{radical} remains valid if the hypothesis that $G$ is a group having an ascending normal series whose quotients are either locally soluble or locally finite is replaced by the hypothesis that $G$ is a group in which every finitely generated subgroup has an ascending normal series whose quotients are either locally soluble or locally finite (in that proposition the group $G$ is finitely generated anyway). Thus, it becomes evident that a group $G$ belongs to $\mathcal V(m,n,v)$ if and only if the $w$-values in $G$ are $n$-Engel and every finitely generated subgroup of $v(G)$ has an ascending normal series whose quotients are either locally soluble or locally finite. So we obtain an alternative characterization of the variety $\mathcal V(m,n,v)$.
\medskip

\noindent {\it The variety $\mathcal V(m,n,v)$ is precisely the class of all groups $G$ in which the $w$-values are $n$-Engel and every finitely generated subgroup of $v(G)$ has an ascending normal series whose quotients are either locally soluble or locally finite.}
\medskip

Similar remarks can be made vis-\`a-vis Theorem C. Let $m,n$ be positive integers, $u$ a multilinear commutator word and $v$ the product of 896 $u$-words. Set $w=v^m$. Denote by $\mathcal U(m,n,u)$ the class of all groups $G$ in which the $w$-values are $n$-Engel and the verbal subgroup $u(G)$ has an ascending normal series whose quotients are either locally soluble or locally finite. Theorem C asserts that the class $\mathcal U(m,n,u)$ is a variety. Applying considerations similar to the above we obtain the following characterizations of the variety $\mathcal U(m,n,u)$.
\medskip

\noindent {\it The variety $\mathcal U(m,n,u)$ is precisely the class of all groups $G$ in which the $w$-values are $n$-Engel and the verbal subgroup $u(G)$ is (locally nilpotent)-by-(locally finite).}
\medskip

\noindent {\it The variety $\mathcal U(m,n,u)$ is precisely the class of all groups $G$ in which the $w$-values are $n$-Engel and every finitely generated subgroup of $u(G)$ has an ascending normal series whose quotients are either locally soluble or locally finite.}
\medskip

Finally, in the same spirit, we remark that the variety handled in Theorem A can be characterized as the class of all groups $G$ in which the $w$-values are $n$-Engel and every finitely generated subgroup of $w(G)$ has an ascending normal series whose quotients are either locally soluble or locally finite.


\begin{thebibliography}{10}
\bibitem{BSTT} R. Bastos, P. Shumyatsky, A. Tortora and M. Tota, {\it On groups admitting a word whose values are Engel}, Int. J. Algebra Comput. {\bf 23} (2013), no. 1, 81--89.

\bibitem{bs} R. Bastos, P. Shumyatsky, {\it On profinite groups with Engel-like conditions},
J. Algebra {\bf 427} (2015), 215--225.

\bibitem{GM} R. Guralnick and G. Malle, {\it Variations on the Baer-Suzuki theorem}, Math. Z., to appear, preprint available at {\tt arXiv:1310.5909 [math.GR]}.

\bibitem{mhall} M. Hall, Jr., {\it The Theory of Groups}, The Macmillan Co., New York, 1959.

\bibitem{hahi} P. Hall and G. Higman, {\it On the $p$-length of $p$-soluble groups and reduction theorems for Burnside's problem}, Proc. London Math. Soc. (3) {\bf 6}  (1956),  1--42.

\bibitem{KS} E.\,I. Khukhro and P. Shumyatsky, {\it Nonsoluble and non-$p$-soluble length of finite groups}, preprint available at {\tt arXiv:1310.2434v3 [math.GR]}.

\bibitem{nise} N. Nikolov and D. Segal, {\it On finitely generated profinite groups, I: strong completeness and uniform bounds}, Annals of Mathematics (2) {\bf 165} (2007), 171--238.

\bibitem{Pl} B.\,I. Plotkin, {\it Radicals and nil-elements in groups}, Izv. Vys$\check{\rm s}$ U$\check{\rm c}$ebn. Zaved. Matematika {\bf 1} (1958), 130--135.

\bibitem{Rob} D.\,J.\,S. Robinson, \textit{A course in the theory of groups}, 2nd edition, Springer-Verlag, New York, 1996.

\bibitem{shu1} P. Shumyatsky, {\it A (locally nilpotent)-by-nilpotent variety of groups}, Math. Proc. Cambridge Philos. Soc. {\bf 132} (2002), no. 2, 193--196.

\bibitem{shu2} P. Shumyatsky, {\it On varieties arising from the solution of the Restricted Burnside Problem},
J. Pure Appl. Algebra {\bf 171} (2002), no. 1, 67--74.

\bibitem{shu3} P. Shumyatsky, {\it Elements of prime power order in residually finite groups}, Int. J. Algebra Comput. {\bf 15} (2005), no. 3, 571--576.

\bibitem{shu4} P. Shumyatsky, {\it Multilinear commutators in residually finite groups}, Israel J. Math. {\bf 189} (2012), 207--224.

\bibitem{STT2} P. Shumyatsky, A. Tortora and M. Tota, {\it On locally graded groups with a word whose values are Engel}, Proc. Edinburgh Math. Soc., to appear, preprint available at {\tt arXiv:1305.3045v2 [math.GR]}.

\bibitem{STT3} P. Shumyatsky, A. Tortora and M. Tota, {\it An Engel condition for orderable groups}, Bull. Braz. Math. Soc. (N.S.), to appear, preprint available at {\tt arXiv:1402.5247 [math.GR]}.

\bibitem{w} J.\,S. Wilson, {\it Two-generator conditions for residually finite groups}, Bull. London Math. Soc. {\bf 23} (1991), 239--248.

\bibitem{ze0} E.\,I. Zelmanov,  {\it Some problems in the theory of groups and Lie algebras}, Math. USSR-Sb. {\bf 66} (1990),  no. 1, 159--168.

\bibitem{ze1} E.\,I. Zelmanov,  {\it Solution of the restricted
Burnside problem for groups of odd exponent}, Math. USSR-Izv.
{\bf 36} (1991), no. 1, 41--60.

\bibitem{ze2} E.\,I. Zelmanov, {\it Solution of the restricted
Burnside problem for $2$-groups}, Math. USSR-Sb. {\bf 72} (1992), no. 2, 543--565.

\end{thebibliography}
\end{document}